\documentclass[amstex,12pt,reqno]{amsart}
\usepackage{amsmath, amssymb, amsthm}
\usepackage{mathrsfs}

\usepackage{mathtools}
\mathtoolsset{showonlyrefs=true}

\newtheorem{theorem}{Theorem}[section]
\newtheorem{thm}{Theorem}
\newtheorem{coro}[thm]{Corollary}
\newtheorem{lemma}[theorem]{Lemma}
\newtheorem{proposition}[theorem]{Proposition}
\newtheorem{corollary}[theorem]{Corollary}

\theoremstyle{definition}

\newtheorem{remark}[theorem]{Remark}

\numberwithin{equation}{section}

\title[Teichm\"uller space of 
circle diffeomorphisms]{The complex structure of the Teichm\"uller space of 
circle diffeomorphisms in the Zygmund smooth class II}
\author[K. Matsuzaki]{Katsuhiko Matsuzaki}
\address{Department of Mathematics, School of Education, Waseda University \endgraf
Shinjuku, Tokyo 169-8050, Japan}
\email{matsuzak@waseda.jp}

\makeatletter
\@namedef{subjclassname@2020}{%
\textup{2020} Mathematics Subject Classification}
\makeatother

\thanks{Research supported by 
Japan Society for the Promotion of Science (KAKENHI 23K25775 and 23K17656)}

\begin{document}

\maketitle

\begin{abstract}
In our previous paper with the same title,  
we established the complex Banach manifold structure for the Teich\-m\"ul\-ler space of  
circle diffeomorphisms whose derivatives belong to the Zygmund class.  
This was achieved by demonstrating that the Schwarzian derivative map is a holomorphic split submersion.  
We also obtained analogous results for the pre-Schwarzian derivative map.  

In this second part of the study,  
we investigate the structure of the image of the pre-Schwarzian derivative map,  
viewing it as a fiber space  
over the Bers embedding of the Teich\-m\"ul\-ler space,  
and prove that it forms a real-analytic disk-bundle.  
Furthermore, we consider the little Zygmund class and establish corresponding results  
for the closed Teich\-m\"ul\-ler subspace consisting of mappings in this class.  

Finally, we construct the quotient space of this subspace in analogy with the asymptotic  
Teich\-m\"ul\-ler space and prove that the quotient Bers embedding and pre-Bers embedding  
are well-defined and injective, thereby endowing it with a complex structure modeled on a quotient Banach space.  
\end{abstract}

\section{Introduction and a Summary of Part I}

This paper continues the study from Part I \cite{M3} on the class of circle diffeomorphisms whose derivatives satisfy the Zygmund continuity condition. 
The corresponding Teich\-m\"ul\-ler space $T^Z$ for these circle diffeomorphisms was defined by Tang and Wu \cite{TW1}. In \cite{M3}, we endowed $T^Z$ with the structure of a complex Banach manifold. In this introduction, we first summarize these results.

We denote by ${\rm Bel}(\mathbb D^*)$ the space of Beltrami coefficients $\mu$, which are measurable functions on the exterior of the unit disk $\mathbb D^*= \{z \mid |z|>1\} \cup \{\infty\}$ with the $L^\infty$ norm $\Vert \mu \Vert_\infty$ less than $1$. Two Beltrami coefficients $\mu, \nu \in {\rm Bel}(\mathbb D^*)$ are said to be \emph{Teich\-m\"ul\-ler equivalent} if the normalized quasiconformal self-homeomorphisms 
$f^\mu$ and $f^\nu$ of $\mathbb D^*$, having $\mu$ and $\nu$ as their complex dilatations, coincide on the unit circle 
$\mathbb S = \partial \mathbb D^*$. The normalization is given by fixing three distinct points on $\mathbb S$. The \emph{universal Teich\-m\"ul\-ler space} $T$ is the quotient space of ${\rm Bel}(\mathbb D^*)$ by the Teich\-m\"ul\-ler equivalence. We denote this projection by 
$\pi: {\rm Bel}(\mathbb D^*) \to T$ and write $\pi(\mu)=[\mu]$.

To endow $T$ with a complex structure, we introduce the complex Banach space $A(\mathbb D)$ of holomorphic functions $\varphi$ on the unit disk $\mathbb D = \{z \mid |z|<1\}$ satisfying
\[
\Vert \varphi \Vert_A=\sup_{|z|<1}\,(1-|z|^2)^2|\varphi(z)|<\infty.
\]  
For every $\mu \in {\rm Bel}(\mathbb D^*)$, let $f_\mu$ be the quasiconformal self-homeomorphism of the extended complex plane $\widehat {\mathbb C}$ with complex dilatation $0$ on $\mathbb D$ and $\mu$ on $\mathbb D^*$. 
Then, we define the \emph{Schwarzian derivative map} $\Phi$ by assigning $\mu$ to 
the Schwarzian derivative $S_{f_\mu|{\mathbb D}}$ of $f_\mu|_{\mathbb D}$. 
It is known that $S_{f_\mu|_{\mathbb D}} \in A(\mathbb D)$ and that $\Phi: {\rm Bel}(\mathbb D^*) \to A(\mathbb D)$ is a holomorphic split submersion onto its image (see \cite[Section 3.4]{Na}).

Moreover, $\Phi$ factors through $\pi$, 
yielding a well-defined injection $\alpha: T \to A(\mathbb D)$ satisfying $\alpha \circ \pi = \Phi$. 
This is called the \emph{Bers embedding} of $T$. By the properties of $\Phi$, 
we conclude that $\alpha$ is a homeomorphism onto its image $\alpha(T) = \Phi({\rm Bel}(\mathbb D^*))$. 
This provides $T$ with the structure of a complex Banach manifold as a domain 
(i.e., a connected open subset) in $A(\mathbb D)$, uniquely determined by requiring that the Teich\-m\"ul\-ler projection $\pi: {\rm Bel}(\mathbb D^*) \to T$ be a holomorphic submersion.

To analyze the Zygmund smooth class, we consider the subspace ${\rm Bel}^Z(\mathbb D^*) = {\rm Bel}(\mathbb D^*) \cap L^Z(\mathbb D^*)$, where
\[
L^Z(\mathbb D^*)=\{ \mu \in L^\infty(\mathbb D^*) \mid \Vert \mu \Vert_Z=
\underset{\ |z|>1}{\mathrm{ess}\sup}\,((|z|^2-1)^{-1}\lor 1)|\mu(z)|<\infty\}.
\]  
We note that ${\rm Bel}^Z(\mathbb D^*)$ is a domain in the Banach space $L^Z(\mathbb D^*)$ with norm $\Vert \mu \Vert_Z$. The Teich\-m\"ul\-ler space $T^Z$ is defined as the quotient space of ${\rm Bel}^Z(\mathbb D^*)$ by the Teich\-m\"ul\-ler equivalence, namely, $T^Z = \pi({\rm Bel}^Z(\mathbb D^*))$. In this setting, we consider the complex Banach space
\[
A^Z(\mathbb D)=\{\varphi \in A(\mathbb D) \mid \Vert \varphi \Vert_{A^Z}=\sup_{|z|<1}\,(1-|z|^2)|\varphi(z)|<\infty\}
\]  
as the corresponding function space. It is evident that $\Vert \varphi \Vert_A \leq \Vert \varphi \Vert_{A^Z}$. We restrict $\Phi$ and $\alpha$ to ${\rm Bel}^Z(\mathbb D^*)$ and $T^Z$, respectively, and denote them by the same notation.

The following theorem, proved in \cite{TW1} and \cite{M3}, establishes the fundamental properties of $T^Z$:

\setcounter{thm}{-1}

\begin{thm}\label{previous}
The following statements hold:
\begin{enumerate}
\item The Schwarzian derivative map $\Phi:{\rm Bel}^Z(\mathbb D^*) \to A^{Z}(\mathbb D)$ is a holomorphic split submersion onto its image $\Phi({\rm Bel}^Z(\mathbb D^*))$.
\item $\Phi({\rm Bel}^Z(\mathbb D^*)) = A^Z(\mathbb D) \cap \Phi({\rm Bel}(\mathbb D^*))$, and this is a domain in $A^Z(\mathbb D)$.
\item The Bers embedding $\alpha:T^Z \to A^Z(\mathbb D)$ is a homeomorphism onto its image $\alpha(T^Z) = \Phi({\rm Bel}^Z(\mathbb D^*))$.
\item The Teich\-m\"ul\-ler space $T^Z$ is endowed with the unique complex Banach manifold structure
such that the Teich\-m\"ul\-ler projection $\pi:{\rm Bel}^Z(\mathbb D^*) \to T^Z$ is a holomorphic submersion.
\end{enumerate}
\end{thm}

Next, we consider a fiber space over  
the universal Teich\-m\"ul\-ler space $\alpha(T) \cong T$, which is realized in the space of pre-Schwarzian derivatives of  
$f_\mu|_{\mathbb D}$. Here, for a locally univalent  
holomorphic function $f:\mathbb D \to \mathbb C$,  
the pre-Schwarzian derivative of $f$ is defined as  
$P_f=(\log f')'$, which satisfies $S_f=(P_f)'-(P_f)^2/2$.  
For $\mu \in {\rm Bel}(\mathbb D^*)$, we normalize $f_\mu$ so that it fixes $\infty$.  
Then, $\psi=\log (f_\mu|_{\mathbb D})'$ is well-defined and belongs to  
the space $B(\mathbb D)$ of all Bloch functions on $\mathbb D$ with norm  
\[
\Vert \psi \Vert_B=\sup_{|z|<1}\,(1-|z|^2)|\psi'(z)|<\infty.
\]
The {\it pre-Schwarzian derivative map}  
$\Psi:{\rm Bel}(\mathbb D^*) \to B(\mathbb D)$ is defined by $\mu \mapsto \log (f_\mu|_{\mathbb D})'$, and  
its image is denoted by $\widetilde{\mathcal T} \subset B(\mathbb D)$. Then, the projection  
$\Lambda:\widetilde{\mathcal T} \to \alpha(T)$ is given by $\Lambda(\psi)=\psi''-(\psi')^2/2$.  
Since $\Lambda \circ \Psi=\Phi$ and $\Phi:{\rm Bel}(\mathbb D^*) \to \alpha(T)$ is a holomorphic split submersion,  
it follows that $\Psi$ is a holomorphic split submersion onto $\widetilde{\mathcal T}$,  
and so is $\Lambda$ onto $\alpha(T)$.  

In \cite{TW1},  
the corresponding space to ${\rm Bel}^Z(\mathbb D^*)$ is given as  
\[
B^Z(\mathbb D)=\{ \psi \in B(\mathbb D) \mid \Vert \psi \Vert_{B^Z}=|\psi'(0)|+\sup_{|z|<1}\,(1-|z|^2)|\psi''(z)|<\infty\}.
\]
We restrict $\Psi$ to ${\rm Bel}^Z(\mathbb D^*)$ and  
consider the map $\Psi:{\rm Bel}^Z(\mathbb D^*) \to B^Z(\mathbb D)$ whose image is denoted by  
$\widetilde{\mathcal T}^Z$.  
Then,  
$\widetilde{\mathcal T}^Z=\widetilde{\mathcal T} \cap B^Z(\mathbb D)$ (see \cite[Theorem 1.1]{TW1}).  
We also consider  
$\Lambda:\widetilde{\mathcal T}^Z \to \alpha(T^Z)$ in the same notation.  


\begin{thm}\label{submersion}  
The maps  
$\Psi:{\rm Bel}^Z(\mathbb D^*) \to \widetilde{\mathcal T}^Z \subset B^Z(\mathbb D)$ and  
$\Lambda:\widetilde{\mathcal T}^Z \to \alpha(T^Z) \subset A(\mathbb D^*)$  
are holomorphic split submersions.  
\end{thm}  

\smallskip  
These results are all contained in \cite{M3}, and the present paper continues from this point.  
In Section \ref{II2}, we first investigate the structure of  
the projection $\Lambda:\widetilde{\mathcal T}^Z \to \alpha(T^Z)$  
more closely and prove the following.  

\begin{thm}\label{bundle}  
The space $\widetilde{\mathcal T}^Z$ is a real-analytic disk-bundle over the Teich\-m\"ul\-ler space $T^Z \cong \alpha(T^Z)$  
with projection $\Lambda$.  
\end{thm}  

\smallskip  
To establish this, we consider the Bers fiber space ({\it universal Teich\-m\"ul\-ler curve}) $\widetilde T$ over $T$.  
We define the corresponding fiber space  
$\widetilde T^Z$ over $T^Z$ and  
prove that $\widetilde{\mathcal T}^Z$ is biholomorphically equivalent to $\widetilde T^Z$  
as a fiber bundle over $T^Z$. Since $\widetilde T^Z$ is  
a real-analytic disk-bundle over $T^Z$, analogous to the case of $\widetilde T \to T$,  
this theorem follows. 

In the remainder of this paper, we focus on a closed subspace of $T^Z$ that is defined in a canonical manner.  
The universal Teich\-m\"ul\-ler space $T$ contains the {\it little subspace} $T_0$, consisting of  
all asymptotically conformal classes, meaning that $T_0$ is defined by  
the Beltrami coefficients $\mu$  
that tend to $0$ near the boundary. Similarly, the little subspace $T^Z_0$ of $T^Z$ is defined by  
the Beltrami coefficients in ${\rm Bel}^Z_0(\mathbb D^*)$, where  
${\rm Bel}^Z_0(\mathbb D^*)$ consists of  
$\mu \in {\rm Bel}^Z(\mathbb D^*)$ satisfying the condition that $(|z|^2-1)^{-1}\mu(z)$  
tends to $0$ near the boundary. That is,  
$T^Z_0=\pi({\rm Bel}^Z_0(\mathbb D^*))$.  
This corresponds to the subspace consisting of all  
normalized self-diffeomorphisms $f$ of $\mathbb S$ whose continuous derivative $f'$ satisfies the  
little Zygmund condition.

In Section \ref{II3}, we establish results  
corresponding to Theorems \ref{previous} and \ref{submersion} for these little subspaces,  
along with the little subspaces $A^Z_0(\mathbb D) \subset A^Z(\mathbb D)$ and $B^Z_0(\mathbb D) \subset B^Z(\mathbb D)$  
of holomorphic functions defined similarly by the vanishing of the norms on the boundary.  
These results are summarized in Theorems \ref{Schwarzian} and \ref{hss}, and Corollary \ref{little}.  
Theorem \ref{bundle} also remains valid for the fiber space $\widetilde{\mathcal T}^Z_0=\Psi({\rm Bel}^Z_0(\mathbb D^*))$  
over $\alpha(T_0^Z)$ under the restriction of $\Lambda$ to $\widetilde{\mathcal T}^Z_0$.

\begin{coro}  
The space $\widetilde{\mathcal T}_0^Z$ is a real-analytic disk-bundle over the Teich\-m\"ul\-ler subspace $T_0^Z \cong \alpha(T_0^Z)$  
with projection $\Lambda$.  
\end{coro}  

\smallskip  
In Section \ref{II4}, we consider the quotient space $T^Z/T^Z_0$,  
which is analogous to the {\it asymptotic Teichm\"ul\-ler space} $AT=T/T_0$.  
We establish the foundation of $T^Z/T^Z_0$ by providing a complex manifold structure modeled on the  
quotient Banach space $A^Z(\mathbb D)/A^Z_0(\mathbb D)$. This is achieved by following the standard approach  
for the asymptotic Teichm\"ul\-ler space,  
where we introduce the quotient Bers embedding $\hat \alpha:T^Z/T^Z_0 \to A^Z(\mathbb D)/A^Z_0(\mathbb D)$  
induced by $\alpha$.  
In Theorem \ref{qBers2}, we prove that the quotient of the Schwarzian derivative map $\Phi$  
remains a holomorphic split submersion, as does the quotient of the pre-Schwarzian derivative map $\Psi$.  
From this, we conclude that  
$\hat \alpha$ is a local homeomorphism onto its image.  

\begin{thm}\label{qBers1}
The quotient Bers embedding $\hat \alpha$
is a local homeomorphism onto its image. Hence, $T^Z/T^Z_0$ admits a complex structure modeled on the quotient Banach space  
$A^Z(\mathbb{D})/A^Z_0(\mathbb{D})$ via $\hat \alpha$.
\end{thm}

Moreover, using existing techniques, we show that $\hat \alpha$ is in fact injective,  
similar to existing examples of other quotient Teichm\"ul\-ler spaces. See Theorem \ref{injective}.

In Section 5, we introduce a somewhat new approach to studying $T^Z/T^Z_0$,  
involving the quotient map of the pre-Bers embedding.  
As mentioned above, the image $\widetilde{\mathcal T}^Z$ of the pre-Schwarzian derivative map $\Psi$ in $B^Z(\mathbb D)$  
forms a fiber space over $T^Z$, and $\Psi$ does not factor through the Teichm\"ul\-ler projection $\pi$.  
However, when taking the quotient, $\Psi$ induces a well-defined map  
$\hat \beta:T^Z/T_0^Z \to B^Z(\mathbb D)/B^Z_0(\mathbb D)$,
which we call the quotient pre-Bers embedding.

\begin{thm}\label{pre-Bers}  
The quotient pre-Bers embedding $\hat \beta$ is  
a biholomorphic homeomorphism onto its image $\widetilde{\mathcal T}^Z/B^Z_0(\mathbb D)$ with respect to the complex Banach structure of $T^Z/T_0^Z$
induced by the quotient Bers embedding $\hat \alpha$.  
\end{thm}  


\section{A Disk-Bundle over the Teichm\"uller Space}\label{II2}

For the remainder of this paper, $f_\mu$ always denotes a quasiconformal self-homeomorphism of $\widehat{\mathbb C}$  
whose complex dilatation is $\mu \in {\rm Bel}(\mathbb D^*)$ on $\mathbb D^*$ and $0$ on $\mathbb D$,  
satisfying the normalization conditions  
$f_\mu(\infty)=\infty$, $f_\mu(0)=0$, and $(f_\mu)'(0)=1$.  
Under this normalization,  
the {\it Bers fiber space} (Teichm\"uller curve) over $T^Z$ is defined as  
\[
\widetilde T^Z=\{([\mu],w) \in T^Z \times \widehat{\mathbb C} 
\mid \mu \in {\rm Bel}^Z(\mathbb D^*),\ w \in f_\mu(\mathbb D^*)\}.
\]  
This is a domain in the product complex manifold $T^Z \times \widehat{\mathbb C}$,  
and the projection $P$ onto the first coordinate is a holomorphic submersion onto the Teichm\"uller space $T^Z$.  

Let $\Theta_w(z)=wz(w-z)^{-1}$ be the M\"obius transformation that sends $w$ to $\infty$ while  
satisfying $\Theta_w(0)=0$ and $(\Theta_w)'(0)=1$.  
Then, we define a map  
\[
E:\widetilde T^Z \to \widetilde{\mathcal T}^Z,\qquad E([\mu],w)=\log (\Theta_w \circ f_\mu|_{\mathbb D})'.
\]

\begin{theorem}\label{holofiber}  
The map $E:\widetilde T^Z \to \widetilde{\mathcal T}^Z$ is a biholomorphic homeomorphism that preserves the fiber structure, that is,  
$\alpha \circ P=\Lambda \circ E$.  
\end{theorem}  

\begin{proof}  
We first show that $E$ is locally bounded.  
For a given $\mu \in {\rm Bel}^Z(\mathbb D^*)$, let $\psi_\infty=E([\mu],\infty)$  
and $\psi=E([\mu],w)$ for any $w \in f_\mu(\mathbb D^*)$.  
Then, a straightforward computation yields  
\begin{align}\label{1st2nd}  
\psi'(z)-\psi'_\infty(z)&  
=\frac{-2(f_{\mu})'(z)}{f_{\mu}(z)-w},\\  
\psi''(z)-\psi''_\infty(z)&=\frac{2(f_{\mu})'(z)^2}{(f_{\mu}(z)-w)^2}  
-\frac{2(f_{\mu})''(z)}{f_{\mu}(z)-w}  
\end{align}  
for $z \in \mathbb D$.  

From \eqref{1st2nd}, we obtain $|\psi'(0)-\psi'_\infty(0)|=2/|w|$, which implies that  
$|\psi'(0)|$, a component of the norm $\Vert \psi \Vert_{B^Z}$, is uniformly bounded  
independently of $\mu \in {\rm Bel}^Z(\mathbb D^*)$.  
Moreover, since $|(f_\mu)'(z)|$   
is uniformly bounded for $z \in \mathbb D$ with the bound dependent only on $\Vert \mu \Vert_Z$ 
by \cite[Proposition 6.8]{M1}, we obtain  
\begin{align*}  
&\quad\sup_{|z|<1}(1-|z|^2)|\psi''(z)-\psi''_\infty(z)| \\
&\leq  
\sup_{|z|<1}\left((1-|z|^2)\frac{2|(f_{\mu})'(z)|^2}{|f_{\mu}(z)-w|^2}+  
(1-|z|^2)\frac{2|(f_{\mu})''(z)|}{|f_{\mu}(z)-w|}\right)\\  
&=\sup_{|z|<1}(1-|z|^2)\frac{2|(f_{\mu})'(z)|^2}{|f_{\mu}(z)-w|^2}  
+\sup_{|z|<1}(1-|z|^2)|P_{f_{\mu}}(z)|\frac{2|(f_{\mu})'(z)|}{|f_{\mu}(z)-w|}  
\end{align*}  
which remains bounded when $\mu \in {\rm Bel}^Z(\mathbb D^*)$ and $w \in f_\mu(\mathbb D^*)$ vary locally.  

Under this local boundedness condition,  
if $E$ is G\^{a}teaux holomorphic, then it is holomorphic (see \cite[Theorem 14.9]{Ch}).  
The G\^{a}teaux holomorphy of $E$ follows from the fact that  
for each fixed $z \in \mathbb D$,  
$E([\mu],w)(z)=\log (\Theta_{w} \circ f_{\mu})'(z)$ is  
G\^{a}teaux holomorphic as a complex-valued function (see, e.g., \cite[Lemma 6.1]{WM-2}).  
By the holomorphic dependence of quasiconformal mappings on the Beltrami coefficients,  
this property is verified.  
Thus, we conclude that $E$ is holomorphic on $\widetilde T^Z$.  

It is easy to see that $E$ is bijective.  
The surjectivity stems from the fact that $\log \Theta_w' \circ f_{\mu} \notin \widetilde{\mathcal T}^Z$ if $w \notin f_{\mu}(\mathbb D^*)$.
To prove that $E$ is biholomorphic,  
it remains to show that the derivative $dE$ is surjective at every point of $\widetilde T^Z$.  
Since $d(\alpha \circ P)=d\Lambda \circ dE$ and $\alpha \circ P$ is a submersion onto $\alpha(T^Z)$ with the kernel of the derivative 
along the fiber of $P$,  
we only need to verify that the derivative $dE$ does not vanish in the direction of $w$.  
This follows from  
\[
\frac{\partial E}{\partial w}(z)=\frac{\partial}{\partial w}\log\frac{(f_\mu)'(z)}{(f_\mu(z)-w)^2}=\frac{2}{f_\mu(z)-w} \neq 0.
\]  
Thus, the proof is complete.  
\end{proof} 

For any $\varphi_0 \in \alpha(T^Z)=\Phi({\rm Bel}^Z(\mathbb D^*))$ and  
$\mu_0 \in {\rm Bel}^Z(\mathbb D^*)$ with $\Phi(\mu_0)=\varphi_0$,  
Theorem \ref{previous} (1) implies that  
there exists a holomorphic map $s_{\mu_0}:V_{\varphi_0} \to {\rm Bel}^Z(\mathbb D^*)$  
defined on some neighborhood $V_{\varphi_0} \subset \alpha(T^Z)$ of $\varphi_0$,  
such that $\Phi \circ s_{\mu_0}={\rm id}|_{V_{\varphi_0}}$ and  
$s_{\mu_0}(\varphi_0)=\mu_0$.  
Regarding the regularity of the quasiconformal homeomorphism determined by $\mu_0$,  
\cite[Lemma 6]{M3} implies that $f_{\mu_0}$ is a real-analytic diffeomorphism of $\mathbb D^*$  
if a suitable $\mu_0$ is chosen in the Teichm\"uller equivalence class.  
Moreover, since  
the quasiconformal homeomorphism $f_\mu$ for any $\mu \in s_{\mu_0}(V_{\varphi_0})$ can be  
explicitly expressed as a composition of $f_{\mu_0}$ using a real-analytic bi-Lipschitz quasiconformal reflection  
and by solving the Schwarzian differential equation,  
it follows that $f_\mu$ is also a real-analytic diffeomorphism of $\mathbb D^*$.  
We summarize this conclusion in the following proposition.

\begin{proposition}\label{real-analytic}  
For every Teichm\"uller equivalence class in ${\rm Bel}^Z(\mathbb D^*)$,  
there exists a representative $\mu_0$ such that  
the local holomorphic right inverse $s_{\mu_0}:V_{\varphi_0} \to {\rm Bel}^Z(\mathbb D^*)$ of  
$\Phi:{\rm Bel}^Z(\mathbb D^*) \to A^Z(\mathbb D)$,  
defined on some neighborhood $V_{\varphi_0}$ of $\varphi_0=\Phi(\mu_0)$,  
ensures that  
the quasiconformal homeomorphism $f_\mu$ of $\mathbb D^*$, determined by $\mu=s_{\mu_0}(\varphi)$ for any $\varphi \in V_{\varphi_0}$,  
is a real-analytic diffeomorphism.  
\end{proposition}

\smallskip  
Since $\widetilde T^Z$ and $\widetilde{\mathcal T}^Z$ are biholomorphically equivalent as stated in Theorem \ref{holofiber},  
Theorem \ref{bundle} concerning $\widetilde{\mathcal T}^Z$ follows from the following lemma  
concerning $\widetilde T^Z$.

\begin{lemma}  
The Bers fiber space $\widetilde T^Z$ is a real-analytic disk-bundle over  
$T^Z$ with projection $P$.  
\end{lemma}

\begin{proof}  
For each Teichm\"uller equivalence class in ${\rm Bel}^Z(\mathbb D^*)$,  
we choose $\mu_0$ and $\varphi_0$ as in Proposition \ref{real-analytic}, and  
define the map  
\[
L_{\varphi_0}:V_{\varphi_0} \times \mathbb D^* \to \widetilde T^Z,\qquad  
L_{\varphi_0}(\varphi,\zeta)=(\alpha^{-1}(\varphi), f_{s_{\mu_0}(\varphi)}(\zeta)),  
\]
which satisfies $\alpha \circ P \circ L_{\varphi_0}(\varphi,\zeta)=\varphi$ for every $\varphi \in V_{\varphi_0}$.  
This provides a local trivialization of $\widetilde T^Z$ over $\alpha^{-1}(V_{\varphi_0}) \cong V_{\varphi_0}$.  

To establish that $\widetilde T^Z$ is a real-analytic disk-bundle over $T^Z$,  
it suffices to show that $L_{\varphi_0}$ is a real-analytic diffeomorphism onto  
$(\alpha \circ P)^{-1}(V_{\varphi_0})$.  
Indeed, each $f_{s_{\mu_0}(\varphi)}$ is real-analytic by Proposition \ref{real-analytic},  
and both $f_{s_{\mu_0}(\varphi)}(\zeta)$ and $f_{s_{\mu_0}(\varphi)}^{-1}(w)$ depend holomorphically on $\varphi$  
due to the holomorphic dependence of the solution of the Beltrami equation.  
\end{proof}

\smallskip  
\begin{remark}  
For the Teich\-m\"ul\-ler space $T^\gamma$ of circle diffeomorphisms with $\gamma$-H\"older continuous derivatives  
(see \cite{M1} and \cite{TW2}),  
the same result holds for the fiber space $\widetilde{\mathcal T}^\gamma$ defined analogously.  
\end{remark}

\section{The little Zygmund smooth class}\label{II3}

In this section, we investigate the little class of 
Zygmund smoothness. For a diffeomorphism $f:\mathbb S \to \mathbb C$ onto its image whose derivative $f'$ satisfies the 
Zygmund continuity condition, we say that $f'$ satisfies the {\it little Zygmund condition} if
\[
|f'(e^{i(\theta+t)})-2f'(e^{i\theta})+f'(e^{i(\theta-t)})| =o(t) 
\]
uniformly as $t \to 0$.
For the space ${\rm Bel}^Z(\mathbb D^*)$ of Beltrami coefficients,  
its little subspace is defined as ${\rm Bel}^Z_0(\mathbb D^*)=L^Z_0(\mathbb D^*) \cap {\rm Bel}(\mathbb D^*)$, where 
\[
L^Z_0(\mathbb D^*)=\{ \mu \in L^Z(\mathbb D^*) \mid \lim_{t \to 0}\
\underset{\ 1<|z|<1+t}{\mathrm{ess}\sup}\,(|z|^2-1)^{-1}|\mu(z)|=0\}.
\]
The Teich\-m\"ul\-ler space of the little Zygmund smooth class is defined by
$T^Z_0=\pi({\rm Bel}^Z_0(\mathbb D^*))$.
We note that the notation $_0$ in this paper differs from that in \cite{M1} and \cite{TW2}.

We observe that $L^Z_0(\mathbb D^*)$ is a closed subspace of $L^Z(\mathbb D^*)$, and hence
${\rm Bel}^Z_0(\mathbb D^*)$ is closed in ${\rm Bel}^Z(\mathbb D^*)$.
Indeed, 
if a sequence $\{\mu_k\}_{k \in \mathbb{N}}$ in $L^Z_0(\mathbb D^*)$ and 
$\mu \in L^Z(\mathbb D^*)$ satisfy
$\Vert\mu_k-\mu\Vert_Z\to 0$ as $k\to\infty$, we can show that $\mu \in L^Z_0(\mathbb D^*)$ as follows.
For every $\varepsilon>0$, we choose some $k_0 \in\mathbb{N}$ 
such that $\Vert\mu_{k_0}-\mu\Vert_Z< \varepsilon$. Since $\mu_{k_0} \in L^Z_0(\mathbb D^*)$, 
there exists some $t_0>0$ such that
\[
\underset{\ 1<|z|<1+t_0}{\mathrm{ess}\sup}\,(|z|^2-1)^{-1}|\mu_{k_0}(z)|<\varepsilon.
\]
Then, it follows that
\[
\underset{\ 1<|z|<1+t_0}{\mathrm{ess}\sup}\,(|z|^2-1)^{-1}|\mu(z)| <
\varepsilon+ \Vert\mu_{k_0}-\mu\Vert_Z<2\varepsilon,
\]
which implies that $\mu \in L^Z_0(\mathbb D^*)$.

We claim that $(f^\mu)'$ satisfies the little Zygmund condition if and only if
$\mu$ belongs to ${\rm Bel}^Z_0(\mathbb D^*)$. This follows from the proof of \cite[Theorem 1.1]{TW1}
with a modification using the fact that $(f_\mu)'|_{\mathbb S}$ satisfies the little Zygmund condition 
if and only if 
\[
\lim_{t \to 0} \sup_{1-t<|z|<1}\, (1-|z|^2)|(f_\mu|_{\mathbb D})'''(z)|=0
\]
(see \cite[Theorem 13]{Zy}).

For the spaces $A^Z(\mathbb D)$ and $B^Z(\mathbb D)$ of holomorphic functions,
the little subspaces are similarly defined:
\begin{align*}
A^Z_0(\mathbb D)&=\{\varphi \in A^Z(\mathbb D) \mid \lim_{t \to 0}\sup_{1-t<|z|<1}\,(1-|z|^2)|\varphi(z)|=0\};\\
B^Z_0(\mathbb D)&=\{ \psi \in B^Z(\mathbb D) \mid \lim_{t \to 0}\sup_{1-t<|z|<1}\,(1-|z|^2)|\psi''(z)|=0\}.
\end{align*}
By the same reasoning as above for $L^Z_0(\mathbb D^*)$, we see that
$A^Z_0(\mathbb D)$ and $B^Z_0(\mathbb D)$ are closed subspaces of $A^Z(\mathbb D)$ and $B^Z(\mathbb D)$, respectively.

For the Schwarzian derivative map, these little subspaces correspond appropriately.

\begin{theorem}\label{Schwarzian}
The image of ${\rm Bel}^Z_0(\mathbb D^*)$ under
the Schwarzian derivative map $\Phi$ is contained in $A^Z_0(\mathbb D)$. 
Moreover, $\Phi:{\rm Bel}^Z_0(\mathbb D^*) \to A^Z_0(\mathbb D)$ is a holomorphic split submersion onto its image.
\end{theorem}

\begin{proof}
For every $\mu \in {\rm Bel}^Z_0(\mathbb D^*)$ and $k \in \mathbb N$, let $\mu_k=\mu 1_{\{|z| \geq 1+1/k\}}$.
Then, $\varphi_k=\Phi(\mu_k)$ belongs to $A^Z_0(\mathbb D)$. Indeed, since $f_{\mu_k}$ is conformal on $\{|z|<1+1/k\}$,
$\varphi_k$ is bounded in a neighborhood of $\mathbb S$.
Since $\mu_k \to \mu$ in ${\rm Bel}^Z(\mathbb D^*)$ as $k \to \infty$ and $\Phi$ is continuous, 
we obtain that $\varphi_k=\Phi(\mu_k)$ converges to $\varphi_0=\Phi(\mu)$ in $A^Z(\mathbb D)$. As $A^Z_0(\mathbb D)$ is closed,
$\varphi_0$ belongs to $A^Z_0(\mathbb D)$. 

To establish that $\Phi:{\rm Bel}^Z_0(\mathbb D^*) \to A^Z_0(\mathbb D)$ is a holomorphic split submersion onto its image,
we construct a local holomorphic right inverse $s_{\mu}:V_{\varphi_0} \to {\rm Bel}^Z_0(\mathbb D^*)$ of $\Phi$ defined on a neighborhood  
$V_{\varphi_0} \subset \Phi({\rm Bel}^Z(\mathbb D^*))$ of $\varphi_0=\Phi(\mu)$ 
for every $\mu \in {\rm Bel}^Z_0(\mathbb D^*)$.

We first take $\mu_0 \in {\rm Bel}^Z_0(\mathbb D^*)$ such that $\mu_0$ is Teich\-m\"ul\-ler equivalent to $\mu$ and
$f_{\mu_0}$ is a bi-Lipschitz real-analytic diffeomorphism in the hyperbolic metric 
on $\mathbb D^*$, as in Proposition \ref{real-analytic}. By examining the proof of \cite[Lemma 6]{M3}, we can choose such $\mu_0$ in 
the little space ${\rm Bel}^Z_0(\mathbb D^*)$. Then, as in the proof of \cite[Lemma 7.5]{M1},
we have
\begin{equation}\label{quadbel}
|s_{\mu_0}(\varphi)(z^*)-\mu_0(z^*)| \lesssim (1-|z|^2)^2|\varphi(z)-\varphi_0(z)|
\end{equation} 
for every $\varphi \in V_{\varphi_0}$ and every $z^*=1/\bar z \in \mathbb D^*$.
If $\varphi \in A^Z_0(\mathbb D)$ in addition, then $\varphi-\varphi_0 \in A^Z_0(\mathbb D)$, and
$\mu_0 \in {\rm Bel}^Z_0(\mathbb D^*)$ implies $s_{\mu_0}(\varphi) \in {\rm Bel}^Z_0(\mathbb D^*)$.

For an arbitrary $\mu \in {\rm Bel}^Z_0(\mathbb D^*)$,
we apply the right translations $r_\mu$ and $r_{\mu_0}$, which are biholomorphic automorphisms of  
${\rm Bel}^Z(\mathbb D^*)$ by \cite[Lemma 5]{M3}, and in fact of ${\rm Bel}^Z_0(\mathbb D^*)$
as we can see by the proof of \cite[Proposition 4]{M3}.
Then, $s_\mu=r_{\mu}^{-1} \circ r_{\mu_0} \circ s_{\mu_0}$ is a local holomorphic right inverse of $\Phi$ defined on 
$V_{\varphi_0} \cap A^Z_0(\mathbb D)$ 
satisfying $s_\mu(\varphi_0)=\mu$.
\end{proof}

\begin{corollary}\label{little}
The image of $T^Z_0$ under
the Bers embedding $\alpha$ is a domain of $A^Z_0(\mathbb D)$, given by  
\[
\alpha(T^Z_0)=\Phi({\rm Bel}^Z_0(\mathbb D^*))=A^Z_0(\mathbb D) \cap \Phi({\rm Bel}(\mathbb D^*)).
\]
Moreover, $\alpha:T^Z_0 \to \alpha(T^Z_0)$ is a homeomorphism.
The Teich\-m\"ul\-ler space $T^Z_0$ is endowed with a complex Banach manifold structure via $\alpha$, such that
$T^Z_0$ is a closed submanifold of $T^Z$.
\end{corollary}

\begin{proof}
The claim that $\Phi({\rm Bel}^Z_0(\mathbb D^*))=A^Z_0(\mathbb D) \cap \Phi({\rm Bel}(\mathbb D^*))$ 
follows from the standard argument using the partial quasiconformal extension of the Ahlfors--Weill type 
to a neighborhood of $\mathbb S$, as given by Becker and Pommerenke \cite[Satz 4]{BP}.
See also \cite[Lemma 4.5]{M1} for its statement
and \cite[Theorem 2.8]{TW1} for its application.
The remaining claims follow directly from Theorem \ref{Schwarzian}.
\end{proof}

For the pre-Schwarzian derivative map, we can assert the similar results. We denote its image by
$\widetilde {\mathcal T}^Z_0=\Psi({\rm Bel}^Z_0(\mathbb D^*))$.

\begin{theorem}\label{hss}
The image of ${\rm Bel}^Z_0(\mathbb D^*)$ under the pre-Schwarzian derivative map $\Psi$ is 
a domain of $B^Z_0(\mathbb D)$, given by  
\[
\widetilde {\mathcal T}^Z_0=\Psi({\rm Bel}^Z_0(\mathbb D^*)) = B^Z_0(\mathbb D) \cap \widetilde {\mathcal T}.
\]
Moreover, both $\Psi:{\rm Bel}^Z_0(\mathbb D^*) \to \widetilde {\mathcal T}^Z_0$ and 
$\Lambda:\widetilde {\mathcal T}^Z_0 \to \alpha(T^Z_0)$ 
are holomorphic split submersions.
\end{theorem}

\begin{proof}
The first claim is established in the same way as for the Schwarzian derivative map $\Phi$, following 
the proofs of Theorem \ref{Schwarzian} and Corollary \ref{little}. 

For the latter claim, we consider the holomorphic map
$\Lambda:B^Z(\mathbb D) \to A^Z(\mathbb D)$, which satisfies $\Lambda \circ \Psi=\Phi$.
It follows from the proof of \cite[Lemma 12]{M3} that $\Lambda(B^Z_0(\mathbb D)) \subset A^Z_0(\mathbb D)$.
We prove more general claim for this inclusion later in Lemma \ref{qlambda}.
Since $\Phi$ is a holomorphic split submersion when restricted to ${\rm Bel}^Z_0(\mathbb D^*)$, 
it follows that both $\Psi$ and $\Lambda$ in the statement are also holomorphic split submersions.
\end{proof}

\smallskip
\begin{remark}
For the Teich\-m\"ul\-ler space $T^\gamma$ of circle diffeomorphisms with $\gamma$-H\"older continuous derivatives,
we can similarly define the little subspace $T^\gamma_0$ by requiring the vanishing of the norm on the boundary.
By the same arguments, the corresponding results hold for this setting as well.
\end{remark}

\section{The quotient Teich\-m\"ul\-ler space}\label{II4}

Following the theory of the asymptotic Teich\-m\"ul\-ler space $AT=T/T_0$ (see \cite{GS}), a similar discussion can be developed 
for the quotient Teich\-m\"ul\-ler space $T^Z/T^Z_0$. 
This quotient is defined by the equivalence relation $\approx$ in $T^Z$ and is equipped with the quotient topology. Here,
$\pi(\mu_1) \approx \pi(\mu_2)$ if $\mu_1 \sim_0 \mu_2$ in ${\rm Bel}^Z(\mathbb D^*)$, 
where $\mu_1 \sim_0 \mu_2$ means that there exists some $\nu \in {\rm Bel}^Z_0(\mathbb D^*)$ satisfying $f^\nu \circ f^{\mu_1}=f^{\mu_2}$.

\begin{lemma}\label{equivalence}
For $\mu_1, \mu_2 \in {\rm Bel}^Z(\mathbb D^*)$, the equivalence $\mu_1 \sim_0 \mu_2$ holds if and only if 
$\mu_1-\mu_2 \in L_0^Z(\mathbb D^*)$.
\end{lemma}

\begin{proof}
Let $\nu \in {\rm Bel}^Z(\mathbb D^*)$ satisfy $f^\nu \circ f^{\mu_1} = f^{\mu_2}$.  
Then, $\nu$ is expressed as
\[
\nu(\zeta) = \frac{\mu_2(z) - \mu_1(z)}{1 - \overline{\mu_1(z)} \mu_2(z)} \frac{\partial f^{\mu_1}(z)}{\overline{\partial f^{\mu_1}(z)}}
\]
for $\zeta = f^{\mu_1}(z)$.  
Since $|\zeta| - 1 \asymp |z| - 1$ by \cite[Theorem 6.4]{M1}, it follows that  
$\nu \in {\rm Bel}^Z_0(\mathbb D^*)$ if and only if  
$\mu_1 - \mu_2 \in L_0^Z(\mathbb D^*)$.
\end{proof}


By this claim, we regard ${\rm Bel}^Z(\mathbb D^*)/\!\!\sim_0={\rm Bel}^Z(\mathbb D^*)/L_0^Z(\mathbb D^*)$ as a domain in 
the quotient Banach space $L^Z(\mathbb D^*)/L_0^Z(\mathbb D^*)$. Then, by the definition of
the quotient Teich\-m\"ul\-ler space $T^Z/T^Z_0$, the Teichm\"uller projection $\pi:{\rm Bel}^Z(\mathbb D^*) \to T^Z$
descends to the quotient Teichm\"uller projection $\hat \pi:{\rm Bel}^Z(\mathbb D^*)/\!\!\sim_0 \to T^Z/T^Z_0$, which is continuous.

Next, we consider the quotient of
the Schwarzian derivative map $\Phi:{\rm Bel}^Z(\mathbb D^*) \to A^Z(\mathbb D^*)$ satisfying
$\Phi({\rm Bel}^Z_0(\mathbb D^*)) \subset A^Z_0(\mathbb D^*)$ by Theorem \ref{Schwarzian}.

\begin{lemma}\label{welldef}
If $\mu_1 \sim_0 \mu_2$ for $\mu_1, \mu_2 \in {\rm Bel}^Z(\mathbb D^*)$, then $\Phi(\mu_1)-\Phi(\mu_2) \in A^Z_0(\mathbb D)$.
\end{lemma}

\begin{proof}
By Lemma \ref{equivalence}, if $\mu_1 \sim_0 \mu_2$, then there exists $\nu \in L_0^Z(\mathbb D^*)$ such that
$\mu_1-\mu_2=\nu$. We prove that the derivative $d_{\mu_2+t\nu} \Phi(\nu)$ of $\Phi$ at $\mu_2+t\nu$ for any 
$t \in [0,1]$ in the direction of $\nu$ belongs to $A_0^Z(\mathbb D)$. Then, by
\[
\Phi(\mu_1)-\Phi(\mu_2)=\int_0^1 d_{\mu_2+t\nu} \Phi(\nu) dt,
\]
we obtain $\Phi(\mu_1)-\Phi(\mu_2) \in A^Z_0(\mathbb D)$.

For any $\mu \in {\rm Bel}^Z(\mathbb D^*)$,
let $g:\mathbb D^* \to \Omega^*=f_\mu(\mathbb D^*)$ 
be the welding conformal homeomorphism satisfying that $f^\mu=g^{-1} \circ f_\mu$ on $\mathbb S$.
For any $\nu \in L_0^Z(\mathbb D^*)$, 
the integral representation of the derivative $d_\mu \Phi(\nu)$ (see \cite[p.25]{TT}) is given by
\[
(f_{\mu})_*(d_\mu \Phi(\nu))(\zeta)
=-\frac{6}{\pi} \int_{\Omega^*} \frac{g_*(\nu)(w)}{(w-\zeta)^4}dudv 
\quad (w=u+iv \in \Omega^*,\ \zeta \in \Omega=f_\mu(\mathbb D)),
\]
where 
\begin{align*}
(f_{\mu})_*(d_\mu \Phi(\nu))(\zeta)&=d_\mu \Phi(\nu)(f_{\mu}^{-1}(\zeta))(f_{\mu}^{-1})'(\zeta)^2;\\ 
g_*(\nu)(w)&=\nu(g^{-1}(w))\overline{(g^{-1})'(w)}(g^{-1})'(w)^{-1}. 
\end{align*}
We note that 
$|(f_{\mu}^{-1})'(\zeta)|$ is bounded away from $0$ by \cite[Proposition 6.8]{M1}. 
In order to show that $d_\mu \Phi(\nu) \in A_0^Z(\mathbb D)$,
we consider
\begin{align}\label{AZ0}
(1-|z|^2)|d_\mu \Phi(\nu)(z)|&\lesssim d(\zeta,\partial \Omega)|(f_{\mu})_*(d_\mu \Phi(\nu))(\zeta)|\\
&\lesssim d(\zeta,\partial \Omega) \int_{\Omega^*} \frac{|\nu(g^{-1}(w))|}{|w-\zeta|^4}dudv 
\end{align}
for $z =f_\mu^{-1}(\zeta) \in \mathbb D$, where $d(\zeta,\partial \Omega)$ is 
the Euclidean distance from $\zeta \in \Omega$ to the boundary.

Let $d(w,\partial \Omega^*)$ denote the Euclidean distance from
$w \in \Omega^*$ to the boundary, and let $\varepsilon(w)=d(w,\partial \Omega^*)^{-1}|\nu(g^{-1}(w))|$. 
We see that $|(g^{-1})'(w)|$ is also bounded
because $g$ is quasiconformally extendable to $\mathbb D$ whose complex dilatation is $(\mu^*)^{-1} \in {\rm Bel}^Z(\mathbb D)$
(see \cite[Proposition 4]{M3}),
where the Beltrami coefficient $\mu^*$ is the reflection of $\mu$ across $\mathbb S$, and $(\mu^*)^{-1}$ denotes the complex dilatation of
the inverse mapping of the normalized quasiconformal self-homeomorphism of $\mathbb D$ determined by $\mu^*$.
Then, since $\nu \in L^Z_0(\mathbb D^*)$, we have $\varepsilon(w) \to 0$ uniformly as
$d(w,\partial \Omega^*) \to 0$. 
We choose $\delta_0>0$ for every $\varepsilon_0>0$ such that if $d(w,\partial \Omega^*) \leq \delta_0$,
then $\varepsilon(w) \leq \varepsilon_0$.  
We note that if $|w-\zeta| \leq d(\zeta,\partial\Omega)+\delta_0$, then $\varepsilon(w) \leq \varepsilon_0$. 

By applying the standard estimate to the integral in \eqref{AZ0}, we obtain
\begin{align}
&\quad \int_{\Omega^*} \frac{|\nu(g^{-1}(w))|}{|w-\zeta|^4}dudv\\
&\leq \int_{d(\zeta,\partial\Omega) \leq |w-\zeta|<d(\zeta,\partial\Omega)+\delta_0} \frac{|\nu(g^{-1}(w))|}{|w-\zeta|^4}dudv
+\int_{d(\zeta,\partial\Omega)+\delta_0 \leq |w-\zeta|} \frac{|\nu(g^{-1}(w))|}{|w-\zeta|^4}dudv\\
&\leq \int_{d(\zeta,\partial\Omega) \leq |w-\zeta|<d(\zeta,\partial\Omega)+\delta_0} \frac{\varepsilon_0}{|w-\zeta|^3}dudv
+\int_{{d(\zeta,\partial\Omega)+\delta_0 \leq |w-\zeta|}} \frac{\Vert \nu \Vert_\infty}{|w-\zeta|^4}dudv\\
&\leq \frac{2\pi \varepsilon_0}{d(\zeta,\partial\Omega)}+\frac{\pi \Vert \nu \Vert_\infty}{(d(\zeta,\partial\Omega)+\delta_0)^2}.\label{divide}
\end{align}
Then, taking the limit as $d(\zeta,\partial\Omega) \to 0$ in \eqref{AZ0} together with 
\eqref{divide}, which follows uniformly from $|z| \to 1$, we conclude that
\[
\lim_{|z| \to 1}(1-|z|^2)|d_\mu \Phi(\nu)(z)| \lesssim \lim_{d(\zeta,\partial\Omega) \to 0}
\left(2\pi \varepsilon_0 +\frac{\pi \Vert \nu \Vert_\infty d(\zeta,\partial\Omega)}{(d(\zeta,\partial\Omega)+\delta_0)^2}\right)
=2\pi \varepsilon_0.
\]
Since $\varepsilon_0>0$ can be chosen arbitrarily small, 
we see that this limit is $0$, which implies that $d_\mu \Phi(\nu) \in A^Z_0(\mathbb D)$.
\end{proof}

Lemma \ref{welldef} ensures that $\Phi:{\rm Bel}^Z(\mathbb D^*) \to A^Z(\mathbb D)$ descends to a well-defined map
\[
\widehat \Phi:{\rm Bel}^Z(\mathbb D^*)/\!\!\sim_0={\rm Bel}^Z(\mathbb D^*)/L_0^Z(\mathbb D^*) \to A^Z(\mathbb D^*)/A^Z_0(\mathbb D^*)
\]
which we call the {\it quotient Schwarzian derivative map}.
Then, we obtain the following fundamental property.

\begin{theorem}\label{qBers2}
The quotient Schwarzian derivative map $\widehat \Phi$ is a holomorphic split submersion onto its image $\alpha(T^Z)/A^Z_0(\mathbb D^*)$.
\end{theorem}

\begin{proof}
To see that $\widehat \Phi$ is holomorphic,
we have to check that its derivative $d\Phi$ descends to the quotient spaces.
As shown in the proof of Lemma \ref{welldef}, if $\nu \in L^Z_0(\mathbb D^*)$, 
then $d_\mu \Phi(\nu) \in A^Z_0(\mathbb D)$. Since $A^Z_0(\mathbb D)$ is a closed subspace of $A^Z(\mathbb D)$, 
the derivative $d\Phi$ composed with the quotient projection $L^Z(\mathbb D^*) \to L^Z(\mathbb D^*)/L^Z_0(\mathbb D^*)$
factors into the derivative of $\widehat \Phi$. 
Thus, $\widehat \Phi$ is holomorphic.

To see that $\widehat \Phi$ is a split submersion,
we analyze the holomorphic split submersion $\Phi:{\rm Bel}^Z(\mathbb D^*) \to A^Z(\mathbb{D})$ onto its image more closely by considering the local holomorphic right inverse  
\[
s_{\mu}:V_\varphi \to {\rm Bel}^Z(\mathbb D^*)
\]
of $\Phi$, which is defined on a neighborhood $V_\varphi \subset \Phi({\rm Bel}^Z(\mathbb D^*))$ of $\varphi=\Phi(\mu)$ for every $\mu \in {\rm Bel}^Z(\mathbb D^*)$. In the proof of Theorem \ref{Schwarzian}, we considered the case where $\mu \in {\rm Bel}^Z_0(\mathbb D^*)$ and showed that the image of $V_\varphi \cap A^Z_0(\mathbb{D})$ under $s_{\mu}$ is contained in ${\rm Bel}^Z_0(\mathbb D^*)$. In fact, this result can be generalized to the claim that for any $\mu \in {\rm Bel}^Z(\mathbb D^*)$, the image of $V_\varphi \cap (\varphi+A^Z_0(\mathbb{D}))$ under $s_{\mu}$ is contained in $\mu+L^Z_0(\mathbb D^*)$. 

Thus, the quotient map  
\[
\hat s_{\mu}:V_\varphi/A^Z_0(\mathbb{D}) \to {\rm Bel}^Z(\mathbb D^*)/\!\!\sim_0
\]
is induced from $s_{\mu}$, which turns out to be a local holomorphic right inverse of $\widehat \Phi$. This implies that $\widehat \Phi$ is a holomorphic split submersion onto its image.
\end{proof}


Moreover, the {\it quotient Bers embedding}  
\[
\hat \alpha:T^Z/T_0^Z \to A^Z(\mathbb D)/A^Z_0(\mathbb D)
\]
is well-defined by $\hat \alpha=\widehat \Phi \circ \hat \pi^{-1}$. This is
is a local homeomorphism onto its image because 
$\hat \pi$ is a quotient projection and $\widehat \Phi$ has a continuous local right inverse. 
Thus, $T^Z/T^Z_0$ is a Banach manifold modeled on $A^Z(\mathbb{D})/A^Z_0(\mathbb{D})$ as stated in Theorem \ref{qBers1}.



In fact, we can prove that $\hat \alpha$ is injective using the same method as in \cite{M0}.

\begin{theorem}\label{injective}
The quotient Bers embedding $\hat{\alpha}: T^Z/T^Z_0 \to A^Z(\mathbb{D})/A^Z_0(\mathbb{D})$ is a homeomorphism onto 
its image $\alpha(T^Z)/A^Z_0(\mathbb D^*)$.
\end{theorem}

\begin{proof}
Since $\hat{\alpha}$ is a local homeomorphism onto its image by Corollary \ref{qBers1}, it suffices to show that $\hat{\alpha}$ is injective. Suppose that $\alpha(\pi(\mu))-\alpha(\pi(\nu))=\Phi(\mu)-\Phi(\nu)$ belongs to $A^Z_0(\mathbb{D})$ for $\pi(\mu), \pi(\nu) \in T^Z$. Under this assumption, we prove that $\pi(\mu) \approx \pi(\nu)$, that is, $\mu \sim_0 \nu$ in ${\rm Bel}^Z(\mathbb D^*)$. 

By the method in \cite[Proposition 3.3]{M0}, we can find $\mu_0, \mu_1 \in {\rm Bel}(\mathbb D^*)$ such that
\[
f^\mu=f^{\mu_0} \circ f^{\mu_1} \circ f^\nu,
\]
where $\mu_0$ has compact support in $\mathbb D^*$ and $\mu_1$ is obtained using the local right inverse $s_\nu$ of the Schwarzian derivative map $\Phi$ with a property similar to \eqref{quadbel}. Instead of the barycentric extension used in \cite{M0}, the argument in the proof of Theorem \ref{Schwarzian} ensures the existence of such a map $s_\nu$. Then, using similar arguments as in \cite[Theorem 5.3]{M0}, we can show that $\mu_1 \in {\rm Bel}^Z_0(\mathbb D^*)$. We remark that Lemma \ref{welldef} is applied at appropriate points. Since $\mu_0$ has compact support, the complex dilatation of $f^{\mu_0} \circ f^{\mu_1}$ belongs to ${\rm Bel}^Z_0(\mathbb D^*)$. Thus, we conclude that $\mu \sim_0 \nu$.
\end{proof}

\begin{remark}
For the Teich\-m\"ul\-ler space $T^\gamma$ of circle diffeomorphisms with $\gamma$-H\"older continuous derivatives, the quotient Teich\-m\"ul\-ler space $T^\gamma/T^\gamma_0$ can also be embedded into the quotient Banach space via the quotient Bers embedding. The proof follows the same approach.
\end{remark}

\section{The quotient pre-Bers embedding}

Finally, we consider the quotient of the holomorphic split submersion $\Lambda:\widetilde{\mathcal T}^Z \to \alpha(T^Z)$.

\begin{lemma}\label{qlambda}
If $\psi_1-\psi_2 \in B^Z_0(\mathbb D)$ for $\psi_1, \psi_2 \in B^Z(\mathbb D)$, then
$\Lambda(\psi_1)-\Lambda(\psi_2) \in A^Z_0(\mathbb D)$.
\end{lemma}

\begin{proof}
We show that $\psi_1'\psi_2' \in A^Z_0(\mathbb D)$ for any $\psi_1, \psi_2 \in B^Z(\mathbb D)$.
Indeed, by \cite[Proposition 8]{Zh}, we have
\begin{align*}
(1-|z|^2)|\psi_1'(z)\psi_2'(z)|
&=(1-|z|^2)^{\frac{1}{3}}\cdot(1-|z|^2)^{\frac{1}{3}}|\psi_1'(z)|\cdot(1-|z|^2)^{\frac{1}{3}}|\psi_2'(z)|\\
&\lesssim (1-|z|^2)^{\frac{1}{3}} \Vert \psi_1 \Vert_{B^Z} \Vert \psi_2 \Vert_{B^Z}.
\end{align*}
This, in particular, implies that $\Lambda(B^Z_0(\mathbb D)) \subset A^Z_0(\mathbb D)$ because
\[
\lim_{|z| \to 1}(1-|z|^2)|\Lambda(\psi)| \leq
\lim_{|z| \to 1}(1-|z|^2)|\psi''(z)|+\lim_{|z| \to 1}\frac{1}{2}(1-|z|^2)|\psi'(z)^2|=0
\]
for every $\psi \in B^Z_0(\mathbb D)$.
If $\psi_1-\psi_2 \in B^Z_0(\mathbb D)$, then $\Lambda(\psi_1-\psi_2) \in A^Z_0(\mathbb D)$. Finally, the equality
\[
\Lambda(\psi_1)-\Lambda(\psi_2)=\Lambda(\psi_1-\psi_2) -(\psi_1-\psi_2)'\psi_2'
\]
shows that this belongs to $A^Z_0(\mathbb D)$.
\end{proof}

It follows from this lemma that
the quotient map
\[
\widehat \Lambda:\widetilde{\mathcal T}^Z/B^Z_0(\mathbb D) \to \alpha(T^Z)/A^Z_0(\mathbb D) =
\hat \alpha(T^Z/T_0^Z) \subset A^Z(\mathbb D)/A^Z_0(\mathbb D)
\]
is well-defined.
We now show that this map is biholomorphic.

\begin{theorem}\label{homeo}
The quotient map $\widehat \Lambda$ defined on 
$\widetilde{\mathcal T}^Z/B^Z_0(\mathbb D)$
is a biholomorphic homeomorphism onto $\alpha(T^Z)/A^Z_0(\mathbb D) =\hat \alpha(T^Z/T_0^Z)$. 
\end{theorem}

\begin{proof}
For every $\varphi \in \alpha(T^Z)$ with $\varphi=\Phi(\mu)$, an element $\psi \in \Lambda^{-1}(\varphi)$ in the fiber of
the projection $\Lambda:\widetilde{\mathcal T}^Z \to \alpha(T^Z)$ is represented as 
\[
\psi=\log\, (\Theta_w \circ f_{\mu}|_{\mathbb D})'
\]
for the M\"obius transformation $\Theta_w$ with $w \in f_{\mu}(\mathbb D^*)$. This fact was used in the proof of Theorem \ref{holofiber}.
Then, any element in the fiber is obtained by adding $\log \Theta_w' \circ f_{\mu}$ to $\log f_{\mu}'$. Hence, it suffices to show that
$\psi=\log \Theta_w' \circ f_{\mu}$ belongs to $B^Z_0(\mathbb D)$.

As in the computation of \eqref{1st2nd}, we have
\[
\psi''(z)=\frac{2f_\mu'(z)}{(f_\mu(z)-w)^2}-\frac{2f_\mu''(z)}{f_\mu(z)-w} \quad (z \in \mathbb D).
\]
Here, $|f_\mu(z)-w|$ is bounded away from zero for $w \in f_{\mu}(\mathbb D^*)$, and $|f_\mu'(z)|$ is bounded as shown earlier.
Then, using
\[
(1-|z|^2)|f_\mu''(z)|= (1-|z|^2)|f_\mu'(z)||P_{f_\mu}(z)| \lesssim (1-|z|^2) |P_{f_\mu}(z)|,
\]
and the fact $\pi(\mu) \in T^Z \subset T_0$ which implies  
$\lim_{t \to 0} \sup_{1-t<|z|<1}\, (1-|z|^2) |P_{f_\mu}(z)|=0$,
we obtain
\[
\lim_{t \to 0} \sup_{1-t<|z|<1}\, (1-|z|^2)|\psi''(z)|=0.
\]
Thus, $\psi$ belongs to $B^Z_0(\mathbb D)$. This completes the proof of the bijectivity of $\widehat \Lambda$.

For the holomorphy of $\widehat \Lambda$, we consider the derivative $d\Lambda$ of the holomorphic map $\Lambda$. 
At every point $\phi \in \widetilde{\mathcal T}^Z$ in the direction of $\psi \in B_0^Z(\mathbb D)$, a straightforward computation
yields
\[
d_{\phi}\Lambda(\psi)=\psi''-\phi'\psi'.
\]
By the proof of Lemma \ref{qlambda}, we see that $d_{\phi}\Lambda(\psi)$ belongs to $A_0^Z(\mathbb D)$.
Thus, the derivative $d\Lambda$ descends to the quotient spaces, showing that $\widehat \Lambda$ is holomorphic.

By the inverse mapping theorem, if the derivative $d\widehat \Lambda$ is surjective at every point, then $\widehat \Lambda^{-1}$
is holomorphic. However, since $\Lambda$ is a submersion, the derivative $d\Lambda$ is surjective. Then, the quotient 
$d\widehat \Lambda$ of $d\Lambda$ is also surjective.
\end{proof}

We now consider the {\it quotient pre-Schwarzian derivative map}
\[
\widehat \Psi:{\rm Bel}^Z(\mathbb D^*)/\!\!\sim_0 \to B^Z(\mathbb D)/B^Z_0(\mathbb D).
\]
This map is defined analogously to $\widehat \Phi$ using arguments similar to those in Lemma \ref{welldef},
and it satisfies properties parallel to those in Theorem \ref{qBers2}.
However, since we have already established that $\widehat \Lambda$ is biholomorphic, these properties are no longer necessary.
The quotient pre-Schwarzian derivative map is given by
$\widehat \Psi=\widehat \Lambda^{-1} \circ \widehat \Phi$.
Thus, it inherits the same properties as $\widehat \Phi$.

\begin{corollary}
The quotient pre-Schwarzian derivative map $\widehat \Psi$ is a holomorphic split submersion
onto its image $\widetilde{\mathcal T}^Z/B^Z_0(\mathbb D)$.
\end{corollary}

The pre-Schwarzian derivative map $\Psi:{\rm Bel}^Z(\mathbb D^*) \to B^Z(\mathbb D)$ 
does not factor through the Teichm\"uller projection $\pi$
due to the fiber structure of $\widetilde{\mathcal T}^Z$. 
Instead, we consider these maps in the quotient spaces. This leads to the definition of the 
{\it quotient pre-Bers embedding}
\[
\hat \beta: T^Z/T^Z_0 \to B^Z(\mathbb D)/B^Z_0(\mathbb D),
\]
which is obtained by factorizing $\widehat \Psi$ through $\hat \pi$.
This factorization is possible because
\[
\widehat \Lambda^{-1} \circ \hat \alpha \circ \hat \pi
=\widehat \Lambda^{-1} \circ \widehat \Phi=\widehat \Psi,
\]
which implies that 
$\hat \beta=\widehat \Lambda^{-1} \circ \hat \alpha$.

The quotient Teich\-m\"ul\-ler space $T^Z/T^Z_0$ is endowed with two complex structures:  
one induced by the quotient Bers embedding $\hat \alpha$ and the other by the quotient pre-Bers embedding $\hat \beta$.  
Since both embeddings are biholomorphic onto their images under $\widehat \Lambda$, these structures are equivalent.
This completes the proof of Theorem \ref{pre-Bers}. 

\begin{remark}
The quotient pre-Bers embedding exists for other Teichm\"uller spaces, including the universal Teichm\"uller space $T$ and the Teich\-m\"ul\-ler space $T^\gamma$ for $\gamma$-H\"older continuous derivatives. In particular, the asymptotic Teichm\"uller space $AT=T/T_0$ is embedded into $B(\mathbb D)/B_0(\mathbb D)$ via $\hat \beta$. The proof is simpler than the arguments presented above.
\end{remark}

\end{document}